\documentclass{amsart}
\usepackage{amsfonts,amssymb, wasysym}
\usepackage{multicol}
\usepackage{hyperref}
\usepackage{graphicx}
\hypersetup{colorlinks,citecolor=blue,linkcolor=blue}

\newtheorem{theorem}{Theorem}
\newtheorem{lemma}{Lemma}

\newcommand{\integers}{{\mathbb Z}}
\newcommand{\realnos}{{\mathbb R}}

\begin{document}

\title{Cusp transitivity in hyperbolic 3-manifolds}

\author{John G. Ratcliffe and Steven T. Tschantz}

\address{Department of Mathematics, Vanderbilt University, Nashville, TN 37240
\vspace{.1in}}

\email{j.g.ratcliffe@vanderbilt.edu}

\date{}

\begin{abstract}
In this paper, we study multiply transitive actions of the group of isometries of a cusped finite-volume hyperbolic 3-manifold on the set of its cusps.  In particular, we prove a conjecture of Vogeler that 
there is a largest $k$ for which such $k$-transitive actions exist, and that for each $k \geq 3$, there 
is an upper bound on the possible number of cusps.  
\end{abstract}

\maketitle

\section{Introduction}\label{S:1} 

Let $S$ be a set and $k$ an integer such that $1 \leq k \leq |S|$. 
 An action of a group $G$ on $S$ 
is called $k$-{\it transitive} if for every choice of distinct elements $x_1,\ldots, x_k$ of $S$ and every choice 
of distinct targets $y_1,\ldots,y_k$ in $S$, there is an element $g$ of $G$ such that $gx_i = y_i$ for each $i = 1,\ldots, k$. 
The term {\it transitive} means 1-transitive, and actions with $k > 1$ are called {\it multiply transitive}. 

This paper is concerned with cusped hyperbolic 3-manifolds of finite volume whose group of isometries 
induces a multiply transitive action on the set of cusps of the manifold.  Call such a manifold $k$-transitive if the induced action is $k$-transitive.  All 3-manifolds in this paper are connected. 

Multiply transitive cusped hyperbolic 3-manifolds of finite volume have a high degree of symmetry and 
are therefore are some of the most beautiful cusped hyperbolic 3-manifolds of finite volume. 

Examples of $k$-transitive hyperbolic link complements are described by Vogeler in \cite{V} for $k = 1, 2, 3, 4$; moreover, Vogeler proved that there is no upper bound for the number of cusps for 
a 2-transitive hyperbolic 3-manifold of finite volume.  
Vogeler conjectured that there is a largest $k$ for which $k$-transitive manifolds exist, 
and for each $k \geq 3$ there is an upper bound on the possible number of cusps. 
In this paper, we prove Vogeler's conjecture; more precisely, we prove the following theorem: 

\begin{theorem}\label{T:1} 
Let $M$ be a $k$-transitive cusped hyperbolic 3-manifold of finite volume. 
Then the possible values of $k$ are $1,2, 3, 4$ and $5$; moreover, 
\begin{enumerate}
\item if $M$ is 5-transitive, then $M$ has exactly 5 cusps,  
\item if $M$ is 4-transitive and not 5-transitive, then $M$ has exactly 4 cusps, 
\item if $M$ is 3-transitive and not 4-transitive, then the possible numbers of cusps of $M$ are $3, 5, 6$ and $8$. 
\end{enumerate}
\end{theorem}

\section{Preliminary lemmas}\label{S:2} 

Let $C$ be a cusp of a hyperbolic 3-manifold $M$ of finite volume. 
A {\it link} $L$ of $C$ is a horospherical cross-section of $C$. 
A link $L$ of $C$ has a natural Euclidean metric, and so $L$ is either 
a flat torus or a flat Klein bottle.  We say that $C$ is {\it orientable} if $L$ is a torus 
and $C$ is {\it nonorientable} if $L$ is a Klein bottle. 
If $M$ is orientable, then each cusp of $M$ is orientable.

\begin{lemma} \label{L:1} 
Let $G$ be the group of isometries of a cusped hyperbolic 3-manifold $M$ of finite volume 
that stabilizes a cusp $C$ of $M$. 
Then $G$ acts effectively on each link $L$ of $C$ as a group of isometries with respect 
to the natural Euclidean metric on $L$. 
\end{lemma}
\begin{proof} We shall work in the upper half-space model $U^3$ of hyperbolic 3-space.  
Then there is a discrete subgroup $\Gamma$ of the group $\mathrm{M}(U^3)$ of M\"obius transformations of $\hat E\hbox{}^3$ that leave $U^3$ invariant such that $M = U^3/\Gamma$. 
Let $L$ be a link of the cusp $C$. 
Then there is a point $c$ on the sphere $\hat E\hbox{}^2$ at infinity that is fixed by a parabolic element $f$ 
of $\Gamma$, and there is a maximal open horoball $B$ based at $c$ such that $B$ is precisely invariant 
by the stabilizer $\Gamma_c$ in $\Gamma$ of the point $c$, and 
the inclusion of $B$ into $U^3$ induces a local isometry of $B/\Gamma_c$ onto $C$; moreover, there is a horosphere $\Sigma$ based at $c$ and contained in $B$
such that the inclusion of $B$ into $U^3$ induces an isometry from $\Sigma/\Gamma_c$ to $L$   
where $\Sigma$ has the natural Euclidean metric  (see p.\,139 of \cite{R}) 
and $L$ has the induced Euclidean metric. 
The point $c$ and the horoball $B$ are unique up to the action of $\Gamma$ on $\hat E\hbox{}^2$. 

Let $g$ be an isometry of $M$ that stabilizes the cusp $C$.  
Then $g$ lifts to an element $\tilde g$ of $\mathrm{M}(U^3)$ such that $\tilde g \Gamma \tilde g^{-1} = \Gamma$ by Theorem 13.2.6 of \cite{R}.  
In particular $\tilde g f\tilde g^{-1}$ is a parabolic element of $\Gamma$ that fixes the point $\tilde gc$. 
As $g$ stabilizes the cusp $C$, we have that $\tilde g c = h c$ for some element $h$ of $\Gamma$. 
By replacing $\tilde g$ by $h^{-1}\tilde g$, we may assume that $\tilde g c = c$. 
Now $\tilde g$ is either an elliptic or parabolic transformation by Theorem 5.5.4 of \cite{R}. 
Therefore $\tilde g$ leaves $\Sigma$ invariant and $\tilde g$ 
acts on $\Sigma$ as an isometry by Theorem 4.7.4 of \cite{R}. 
Hence $g$ acts on $L$ as an isometry with respect to the natural Euclidean metric on $L$. 

By conjugating $\Gamma$ in $\mathrm{M}(U^3)$, we may assume that $c = \infty$. 
Then the action of $\tilde g$ on $E^2$, as a Euclidean isometry, determines the action of $\tilde g$ on $U^3$ in a parallel fashion on each horizontal horosphere by Poincar\'e extension (see \S 4.4 of \cite{R}). 
Hence, if $g$ acts trivially on $L$, then $g$ acts trivially on the open subset $C$ of $M$, 
and therefore $g$ is trivial. 
Thus $G$ acts effectively on $L$. 
\end{proof}

\begin{lemma} \label{L:2} 
Let 
$$A = \left(\begin{array}{cc} -1& 0 \\ \phantom{-}0 & 1\end{array}\right), \quad
B = \left(\begin{array}{cc} -1 & 1 \\ \phantom{-}0 & 1 \end{array}\right), \quad 
C= \left(\begin{array}{cc} 0 & 1 \\ 1 & 0\end{array}\right).$$
The group $\mathrm{GL}(2,\integers)$ is the free product of the dihedral group $\langle A, C\rangle$ 
of order $8$ and the dihedral group $\langle B, C\rangle$ of order $12$ amalgamated along 
the dihedral subgroup $\langle -I, C\rangle$ of order $4$.  Every finite subgroup of $\mathrm{GL}(2,\integers)$ is conjugate to a subgroup of $\langle A, C\rangle$ or $\langle B, C\rangle$.
\end{lemma}
\begin{proof}
The matrices $A, B, C$ have order 2, and $(AC)^2 = -I$, and $(BC)^3=-I$. 
Hence $\langle A, C\rangle$ is a dihedral group of order 8,  
and $\langle B, C\rangle$ is a dihedral group of order 12. 

The group $\mathrm{PGL}(2,\integers)$ acts effectively by isometries on the upper complex half-space model of hyperbolic 2-space $H^2$ by
$$\pm\left(\begin{array}{cc} a & b \\ c & c\end{array}\right)z = \frac{az+b}{cz+d} \quad \hbox{if}\  
\det\left(\begin{array}{cc} a & b \\ c & c\end{array}\right) = 1,$$
and 
$$\pm\left(\begin{array}{cc} a & b \\ c & c\end{array}\right)z = \frac{a\overline{z}+b}{c\overline{z}+d} \quad \hbox{if}\  
\det\left(\begin{array}{cc} a & b \\ c & c\end{array}\right) = -1,$$
It is well know that a fundamental domain for the action of  $\mathrm{PGL}(2,\integers)$ on $H^2$ is the hyperbolic triangle $T$ with vertices $i, \frac{1}{2}+\frac{\sqrt{3}}{2}i, \infty$ and corresponding angles $\pi/2, \pi/3, 0$. 
Moreover $\mathrm{PGL}(2,\integers)$ is a triangle reflection group with respect to the reflections  
$R_1(z) = -\overline{z}, R_2(z) = 1/\overline{z}, R_2(z) = 1-\overline{z}$ in the sides of $T$. 
Note that $R_1$ is the reflection in the $y$-axis, $R_2$ is the inversion in the circle $|z| = 1$, 
and $R_3$ is the reflection in the line $x = \frac{1}{2}$. 
Moreover
$$(\pm A)z = R_1(z), \quad (\pm C)z = R_2(z), \quad (\pm B)z = R_3(z).$$

By Poincar{\'e}'s fundamental polygon theorem (Theorem 13.5.3 \cite{R}), 
the group $\mathrm{PGL}(2,\integers)$ has the presentation 
$$\langle R_1, R_2, R_3; R_1^2, R_2^2, R_3^2, (R_1R_2)^2, (R_2R_3)^3\rangle,$$
which is equivalent to the presentation
$$\langle R_1, R_2, R_3, R_4; R_1^2, R_2^2, (R_1R_2)^2, R_2 = R_4, R_3^2, R_4^2, (R_4R_3)^3\rangle.$$
Therefore $\mathrm{PGL}(2,\integers)$ is the free product with amalgamation 
of the dihedral group $\langle R_1, R_2\rangle$ of order 4 and the dihedral group 
$\langle R_2, R_3\rangle$ of order 6 amalgamated along the subgroup $\langle R_2\rangle$ of order 2.

Let $\eta: \mathrm{GL}(2,\integers) \to \mathrm{PGL}(2,\integers)$ the quotient map.  
Then $\mathrm{GL}(2,\integers)$ acts on the Bass-Serre tree of the amalgamated product decomposition of $\mathrm{PGL}(2,\integers)$ via $\eta$. 
Hence $\mathrm{GL}(2,\integers)$ is the amalgamated product of the groups
$$\eta^{-1}(\langle R_1, R_2\rangle) = \langle A, C\rangle\ \ 
\hbox{and}\ \  \eta^{-1}(\langle R_2, R_3\rangle) = \langle B, C\rangle$$
along the subgroup $\eta^{-1}(\langle R_2\rangle) = \langle -I, C\rangle$. 
By the torsion theorem for amalgamated products of groups, every finite subgroup of 
$\mathrm{GL}(2,\integers)$ is conjugate to a subgroup of $\langle A, C\rangle$ or $\langle B, C\rangle$
\end{proof}

\section{On $k$-Transitive Manifolds with $k \geq 5$} 

Our basic reference for the theory of finite permutation groups is Wielandt's classic book \cite{W}. In particular, we refer to \cite{W} for all basic definitions in the elementary theory of permutation groups.  

\begin{lemma}\label{L:3} 
If $G$ is a 4-transitive permutation group of degree $n$ 
with a nontrivial solvable normal subgroup, then $n = 4$ and $G = S_4$.
\end{lemma}
\begin{proof} The group $G$ is primitive by Theorem 9.6 of \cite{W}. 
Let $N$ be a minimal, nontrivial, solvable, normal subgroup of $G$. 
Then $N$ is regular by Theorem 11.5 of \cite{W}. 
Hence $n = 4$ by Theorem 11.3(d) of \cite{W}. 
Therefore $G = S_4$. 
\end{proof}

\begin{theorem} \label{T:2} 
Let $M$ be a cusped hyperbolic 3-manifold $M$ of finite volume such that 
the induced action of the group of isometries of $M$ 
on the set of cusps of $M$ is $k$-transitive with $k \geq 5$.  
Then $k = 5$, and the possible number of cusps of $M$ is $5$; moreover, each cusp of $M$ is orientable.  

\end{theorem}
\begin{proof}
Let $G = {\rm Isom}(M)$.  
It is well known that $G$ is finite (see Theorem 12.7.7 of \cite{R}). 
Let $\Omega$ be the set of cusps of $M$, and let $G^\Omega$ 
be the induced permutation group of $\Omega$. 
Let $n = |\Omega|$.  As $5 \leq k \leq n$, we have that $n \geq 5$. 
As $G^\Omega$ acts transitively on $\Omega$, the degree of $G^\Omega$ is $n$. 
Fix a cusp $C$ of $M$, and let $G_1$ be the subgroup of $G$ that stabilizes $C$,  
and let $G_1^\Omega$ be the corresponding point stabilizer of $G^\Omega$. 

Let $L$ be a link of $C$. 
By Lemma 1, we may regard $G_1$ to be a subgroup of the group $\mathrm{Aff}(L)$ 
of affine transformations of $L$.  
Assume first that $L$ is a torus. 
By Theorems 1 and 3 of \cite{R-T-I}, 
there is a short exact sequence 
$$1\to \realnos^2/\integers^2 \to \mathrm{Aff}(L)\ {\buildrel \pi \over \longrightarrow}\  \mathrm{GL}(2,\integers)\to 1.$$
Then $\pi(G_1)$ is a finite subgroup of $\mathrm{GL}(2,\integers)$. 
Therefore $\pi(G_1)$ is solvable by Lemma 2. 
Hence $G_1$ is solvable, since ${\rm ker}(\pi)$ is abelian. 

The group $G_1$ maps onto $G_1^\Omega$, 
and so $G_1^\Omega$ is solvable. 
Now $G_1^\Omega$ is $(k-1)$-transitive on $\Omega-\{C\}$ by Theorem 9.1 of \cite{W}. 
Hence, the degree of $G_1^\Omega$ is $n-1$. 
As $G_1^\Omega$ is 4 transitive, $G_1^\Omega$ has degree 4 and $G_1^\Omega = S_4$ by Lemma 3. 
Hence $n -1 = 4$, and so $n = 5$. 
Therefore $M$ has five cusps.  Moreover $k = 5$, since $5 \leq k \leq n$. 
For an example of such a manifold, see Example 1 in \S 6.

Now assume that $L$ is a Klein bottle. 
By Theorems 1 and 3 of \cite{R-T-I} and Exercise 9.5.8 of \cite{R}, 
there is a short exact sequence 
$$1\to \realnos/\integers \to \mathrm{Aff}(L)\ {\buildrel \pi \over \longrightarrow}\  (\integers/2\integers)^2 \to 1.$$
Clearly, $G_1$ is solvable. 
The same argument as in the case that $L$ is a torus implies that $G_1^\Omega = S_4$. 

Let $K$ be the kernel of the projection from $\mathrm{Aff}(L)$ to $(\integers/2\integers)^2$. 
Then $K\cap G_1$ is an abelian normal subgroup of $G_1$ that projects onto a normal subgroup $N$ of $G_1^\Omega$ of order 1 or 4, since $S_4$ has a unique abelian proper normal subgroup, and this subgroup has order 4. Therefore $G_1/(K\cap G_1) \cong (G_1K)/K$ has order at most 4 and 
projects onto the group $G_1^\Omega/N$ of order at least 6, which is a contradiction. 
Thus each cusp of $M$ is orientable. 
\end{proof}

\section{On 4-Transitive Manifolds that are not 5-Transitive} 

\begin{lemma} \label{L:4} 
Let $G$ be a 2-transitive permutation group of degree $n$ with 
a nontrivial solvable normal subgroup. 
Then 
\begin{enumerate}
\item The group $G$ has a nontrivial, regular, abelian, normal subgroup $N$. 
\item The group $N$ is an elementary abelian $p$-group  
for some prime number $p$. 
\item The order of $N$ is $n$, and so $n = p^m$ for some positive integer $m$. 
\item The group $N$ is the unique minimal nontrivial normal subgroup of $G$. 
\item The group $N$ is the unique nontrivial abelian normal subgroup of $G$. 
\end{enumerate}
\end{lemma}
\begin{proof}
The group $G$ is primitive by Theorem 9.6 of \cite{W}. 
Let $N$ be a minimal nontrivial solvable normal subgroup of $G$. 
Then $N$ is regular, $N$ is an elementary abelian group, 
and the only nontrivial minimal normal subgroup of $G$ 
by Theorem 11.5 of \cite{W}. 
Hence $n = |N|$ by Proposition 4.2 of \cite{W}. 

Suppose $A$ is a nontrivial abelian normal subgroup of $G$. 
The group $A$ is transitive by Theorem 8.8 of \cite{W}. 
The group $A$ is regular by Proposition 4.4 of \cite{W},  
and so $n = |A|$ by Proposition 4.2 of \cite{W}. 
As the minimal nontrivial normal subgroup of $G$ contained in $A$ is $N$, 
we have that $A = N$. 
\end{proof}

\begin{lemma} \label{L:5} 
Let $G$ be a 3-transitive permutation group of degree $n$ with 
a nontrivial solvable normal subgroup. 
Then either $n = 3$ and $G = S_3$ or $n = 2^m$ for some integer $m \geq 2$.  
\end{lemma}
\begin{proof}
This follows from Lemma 4(1), Theorem 11.3(b) of \cite{W}, and $3\leq n$. 
\end{proof}

\begin{lemma} \label{L:6} 
Let $H$ be a finite subgroup of  $\realnos^n/\integers^n$ for some positive integer $n$.  
Then the rank (minimal number of generators) of $H$ is at most $n$. 
\end{lemma}
\begin{proof}
Let $\eta: \realnos^n \to \realnos^n/\integers^n$ be the quotient map. 
Then $\integers^n$ is a finite index subgroup of $\eta^{-1}(H)$. 
Hence $\eta^{-1}(H)$ is a discrete subgroup of $\realnos^n$. 
Therefore $\eta^{-1}(H)$ is generated by a set of linearly independent vectors 
$\{v_1,\ldots, v_m\}$ with $m \leq n$ by Theorem 5.3.2 of \cite{R}. 
Therefore $H$ is generated by $\{\eta(v_1), \ldots, \eta(v_m)\}$. 
\end{proof}

\begin{theorem}\label{T:3} 
Let $M$ be a cusped hyperbolic 3-manifold of finite volume such that 
the induced action of the group of isometries of $M$ 
on the set of its cusps is $4$-transitive but not $5$-transitive. 
Then the possible number of cusps of $M$ is $4$. 
\end{theorem}
\begin{proof}
We continue with the notation in the proof of Theorem 2.   
Then as in the proof of Theorem 1, we have that $G_1^\Omega$ is solvable 
and $G_1^\Omega$ is 3-transitive on $\Omega-\{C\}$. 
By Lemma 5, the degree of $G_1^\Omega$ is 3 or $2^m$ for some integer $m \geq 2$. 
In the first case, the degree of $G^\Omega$ is 4 by Theorem 9.1 of \cite{W}. 
Therefore $M$ has exactly 4 cusps. 

Now assume that the second case holds. 
Then the degree of $G_1^\Omega$ is at least $4$. 
Hence, the order of $G_1^\Omega$ is at least 24 by Theorem 9.3 of \cite{W}. 
Let $K$ be the kernel of the projection  $\pi: {\rm Aff}(L) \to ({\rm GL}(2,\integers)$ or $(\integers/2\integers)^2$). 
Then $K\cap G_1$ is an abelian normal subgroup of $G_1$ 
that projects to an abelian normal subgroup $N$ of $G_1^\Omega$. 
Assume first that $N$ is trivial. 
Then the finite subgroup $G_1/(K\cap G_1) \cong G_1K/K$ of ${\rm GL}(2,\integers)$ 
or $(\integers/2\integers)^2$ projects onto $G_1^\Omega$. 
Hence, the order of $G_1^\Omega$ is at most 12 by Lemma 2, which is a contradiction. 
Therefore $N$ is nontrivial. 

The group $N$ is an elementary 2-group of rank at least 2 by Lemmas 4 and 5. 
The group $K\cap G_1$ has rank at most 2 by Lemma 6. 
Hence, the rank of $N$ is at most 2. 
Therefore, the rank of $N$ is 2 and the degree of $G_1^\Omega$ is 4. 
Hence,  $G_1^\Omega = S_4$, since $G_1^\Omega$ has order at least 24. 
Therefore $G^\Omega$ is 5-transitive by Theorem 9.1 of \cite{W}. 
However, $G^\Omega$ is not 5-transitive by assumption, 
and so the second case does not occur. Therefore $M$ has exactly 4 cusps. 
For an example of such a manifold, see Example 2 in \S 6. 
\end{proof}

\section{On 3-Transitive Manifolds that are not 4-Transitive} 

\begin{theorem}\label{T:4} 
Let $M$ be a cusped hyperbolic 3-manifold of finite volume such that  
the induced action of the group of isometries of $M$ 
on the set of its cusps is $3$-transitive but not $4$-transitive.   
Then the possible numbers of cusps of 
$M$ are $3, 5, 6$ and $8$,  and 
the corresponding permutation groups are $S_3, A_5, \mathrm{PGL}(2,5)$ and $\mathrm{PGL}(2,7)$  
respectively, of orders $6, 60, 120$ and $336$ respectively. 
\end{theorem}
\begin{proof}
We continue with the notation in the proofs of Theorem 2 and 3.   
Then as in the proof of Theorem 1, we have that $G_1^\Omega$ is solvable  
and $G_1^\Omega$ is 2-transitive on $\Omega-\{C\}$.

Let $K$ be the kernel of the projection  $\pi: {\rm Aff}(L) \to ({\rm GL}(2,\integers)$ or $(\integers/2\integers)^2$). 
Then $K\cap G_1$ is an abelian normal subgroup of $G_1$ 
that projects to an abelian normal subgroup $N$ of $G_1^\Omega$. 

Assume first that $N$ is trivial. 
Then the finite subgroup $G_1/(K\cap G_1) \cong G_1K/K$ of ${\rm GL}(2,\integers)$ 
or $(\integers/2\integers)^2$ projects onto $G_1^\Omega$. 
Hence, $G_1^\Omega$ is isomorphic to a subgroup of a dihedral group of order 8 or 12 by Lemma 2.
Therefore $G_1^\Omega$ is a cyclic group of order $2, 3, 4, 6$ or a dihedral group 
of order $4, 6, 8, 12$. 
By Lemma 4(5), the group $G_1^\Omega$ has a unique nontrivial abelian normal subgroup. 
Hence, $G_1^\Omega$ is either a group of order $2$ or $3$ or a dihedral group of order 6. 

Suppose that $G_1^\Omega$ has order 2. 
Then $G_1^\Omega$ has degree 2, and so the degree of $G^\Omega$ is 3 by Theorem 9.1 of \cite{W}.  
Hence $M$ has 3 cusps and $G^\Omega = S_3$. 

Suppose that $G_1^\Omega$ has order 3. 
Then $G_1^\Omega$ has degree 3. 
Now 6 divides $|G_1^\Omega|$ by Theorem 9.3 of \cite{W}, which is a contradiction.  
Hence, this case does not occur. 

Suppose that $G_1^\Omega$ is a dihedral group of order 6. 
Then $G_1^\Omega$ has degree 3, and so $G_1^\Omega = S_3$. 
Hence, $G_1^\Omega$ is 3-transitive. 
Therefore $G^\Omega$ is 4-transitive by Theorem 9.1 of \cite{W}, 
but $G^\Omega$ is not 4-transitive by assumption, 
and so this case does not occur. 

Now assume that $N$ is nontrivial. 
Then $N$ is an elementary abelian $p$ group of order $p^m = \mathrm{deg}(G_1^\Omega)$ 
for some prime number $p$ and some positive integer $m$ by Lemma 4. 
The order of $G_1^\Omega$ is divisible by $p^m(p^m-1)$ by Theorem 9.3 of \cite{W}. 
The finite group $G_1/(K\cap G_1)\cong G_1K/K$ is isomorphic to a subgroup of $\mathrm{GL}(2,\integers)$ 
or $(\integers/2\integers)^2$, and $G_1/(K\cap G_1)$ projects onto $G_1^\Omega/N$. 
Therefore $p^m - 1 \leq |G_1^\Omega/N| \leq 12$. 
The group $K\cap G_1$ has rank at most 2 by Lemma 6, 
and $K\cap G_1$ projects onto $N$. 
Therefore $m = \mathrm{rank}(N)$ is 1 or 2. 

Assume first that $m = 1$.  
Then $G_1^\Omega$ has degree $p$. 
The solvable 2-transitive finite permutation groups were classified by Huppert \cite{H} 
(see also \S7 of Chapter XII of \cite{H-B}).  Let $F_q$ be the finite field of order $q$. 
According to Huppert \cite{H} and by Theorem 12 in Chapter VII of \cite{L}, the group $G_1^\Omega$ is isomorphic to a subgroup of the group of affine transformation of the vector space $F_p$ over $F_p$,  
$$\Gamma(p) =\{\alpha: F_p \to F_p: \alpha(x) = ax + b\ 
\hbox{with}\ a, b\in F_p\ \hbox{and}\ a\neq 0\}.$$
The group $\Gamma(p)$ is an extension of the additive group $F_p^+$ of the field $F_p$ 
by the multiplicative group $F_p^*$ of $F_p$, that is, we have an short exact sequence 
$$0 \to F_p^+\ {\buildrel \iota \over \longrightarrow}\ \Gamma(p)\ {\buildrel \eta \over \longrightarrow}\ F_p^* \to 1.$$
where $\iota(b) = x + b$ and $\eta(ax+b) = a$. 
Therefore $|\Gamma(p)| = p(p-1)$. 
As $|G_1^\Omega| \geq p(p-1)$, we have that $G_1^\Omega \cong \Gamma(p).$
Hence $G_1^\Omega/N \cong  F_p^*$, and so $G_1^\Omega/N$ is a cyclic group of order $p-1$ by Theorem 11 of Chapter VII of \cite{L}.

The group $G_1/(K\cap G_1)$ is either a finite cyclic group of order at most 6 or 
a finite dihedral group of order at most 12. 
As $G_1/(K\cap G_1)$ projects onto $G_1^\Omega/N$, we have that 
$G_1^\Omega/N$ is a cyclic group of order at most 6. 
Therefore $p = 2, 3, 5$ or $7$. 

Assume that $p = 2$, then $G_1^\Omega$ has order 2. 
Then $G_1^\Omega$ has degree 2, and so the degree of $G^\Omega$ is 3 by Theorem 9.1 of \cite{W}.  
Hence $M$ has 3 cusps and $G^\Omega = S_3$. 
For an example of such a manifold, see Example 3 in \S 6. 

Assume that $p = 3$.  Then $G_1^\Omega$ has degree 3 and $G_1^\Omega = S_3$. 
Hence $G_1^\Omega$ is 3-transitive. 
Therefore $G^\Omega$ is 4-transitive by Theorem 9.1 of \cite{W}. 
However $G^\Omega$ is not 4-transitive by assumption. 
Therefore the case $p = 3$ does not occur. 

Assume that $p = 5$.  Then the degree of  $G^\Omega$ is $6$, 
and so $M$ has 6 cusps.  As $|G^\Omega_1| = 5\cdot 4$, 
we have that $|G^\Omega| = 6\cdot 5 \cdot 4$. 
Therefore $G^\Omega$ is sharply 3-transitive. 
By Theorem 2.6 of Chapter XI of \cite{H-B}, we have that $G^\Omega = \mathrm{PGL}(2,5)$. 
For an example of such a manifold, see Example 4 in \S6. 


Assume that $p = 7$. Then the degree of  $G^\Omega$ is $8$, 
and so $M$ has 8 cusps. As $|G^\Omega_1| = 7\cdot 6$, 
we have that $|G^\Omega| = 8\cdot 7 \cdot 6$. 
Therefore $G^\Omega$ is sharply 3-transitive. 
By Theorem 2.6 of Chapter XI of \cite{H-B}, we have that $G^\Omega = \mathrm{PGL}(2,7)$. 
For an example of such a manifold, see Example 5 in \S 6. 


Now assume that $m = 2$. As $p^2 - 1 \leq 12$, we have that $p = 2$ or 3. 

Assume that $p = 2$.  Then $G_1^\Omega$ is a 2-transitive permutation group of degree $p^2 = 4$
and order a multiple of $p^2(p^2-1) = 12$. Hence $G_1^\Omega$ is either $A_4$ or $S_4$. 
If $G_1^\Omega$ were $S_4$, then $G_1^\Omega$ would be 4-transitive and $G^\Omega$ 
would be 5-transitive, which is contrary to assumption. 
Therefore $G_1^\Omega = A_4$. 
Now $G^\Omega$ has degree 5 and has order $5|G_1^\Omega| = 60$. 
Hence $M$ has 5 cusps and $G^\Omega = A_5$. 
For an example of such a manifold, see Example 6 in \S 6.


The remaining part of the proof is devoted to proving that the case $m = 2$ and $p = 3$ does not occur.  
On the contrary, assume $m = 2$ and $p = 3$. 
Then the order of $G_1^\Omega/N$ is a multiple of $p^2-1 = 8$. 
As $|G_1^\Omega/N| \leq 12$, we have that $|G_1^\Omega/N| = 8$.
Hence $|G_1^\Omega|= p^2(p^2-1) = 72$. 
From Table 7.3 of \cite{C}, we deduce that either $G_1^\Omega$ is isomorphic to a subgroup of 
$$\Gamma(9) = \{\alpha:F_9 \to F_9: \alpha(x) =ax^{3^c} +b\ \hbox{with}\ a, b\in F_9, a\neq 0\ \hbox{and}\ c = 0,1\}.$$
or $|G_1^\Omega/N| \geq |\mathrm{SL}(2,3)|$. 
As $|\mathrm{SL}(2,3)| = 24 > 8$, the latter case is not possible. 
Therefore $G_1^\Omega$ is isomorphic to a subgroup of $\Gamma(9)$. 

Consider the following subgroup of $\Gamma(9)$:
$$\mathrm{\Gamma L}(1,9) = \{\alpha:F_9 \to F_9: \alpha(x) =ax^{3^c}\ \hbox{with}\ a\in F_9^\ast \ \hbox{and}\ c = 0,1\}.$$
We have a short exact sequence
$$0 \to F_9^+\ {\buildrel \iota \over \longrightarrow}\ \Gamma(9)\ {\buildrel \eta \over \longrightarrow}\ \mathrm{\Gamma L}(1,9)\to 1.$$
where $\iota(b) = x + b$ and $\eta(ax^{3^c}+b) = ax^{3^c}$. 
Moreover, we have a short exact sequence 
$$0 \to F_9^*\ {\buildrel \kappa \over \longrightarrow}\ \mathrm{\Gamma L}(1,9)\ {\buildrel \rho\over \longrightarrow}\ \integers/2\integers \to 1$$
where $\kappa(a) = ax$ and $\rho(ax^{3^c}) = c\integers$, since the Frobenius automorphism 
$x^3$ of $F_9$ has order  2. 
Therefore $\Gamma(9)$ is a solvable group of order $9\cdot 8\cdot 2 = 144$. 

The group $\Gamma(9)$ acts effectively on the set $F_9$, by evaluation, as a 2-transitive permutation group. 
Hence $\Gamma(9)$ has a unique, nontrivial, abelian, normal subgroup of order 9 by Lemma 4(5). 
Let $T = \{x+ b: b\in F_9\}$ be the translation subgroup of $\Gamma(9)$. 
Then $T$ has order 9 and $T$ is the unique, nontrivial, abelian, normal subgroup of $\Gamma(9)$. 

Let $\Gamma_0(9)$ be the stabilizer of $0$. 
Then 
$$\Gamma_0(9) =\{ax^{3^c}: a\in F_9^* \ \hbox{and}\ c = 0,1\} = \mathrm{\Gamma L}(1,9).$$
The group $\Gamma_0(9)$ acts transitively on $F_9^*$ by Theorem 9.1 of \cite{W}. 
By Theorem 11.2 of \cite{W}, the induced action of $\Gamma_0(9)$ on $T-\{x\}$ by conjugation 
corresponds to the action of $\Gamma_0(9)$ on $F_9^*$; indeed, 
with respect to multiplication given by composition of polynomial functions, we have that 
$$(ax)(x+b)(ax)^{-1} = (ax)(x+b)(a^{-1}x) = (ax)(a^{-1}x+b) = x + ab,$$
and 
$$(x^3)(x+b)(x^3)^{-1} = (x^3)(x+b)(x^3) = (x^3)(x^3+b)=x+b^3.$$

Let $H$ be a subgroup of $\Gamma(9)$ that is isomorphic to $G_1^\Omega$. 
Then the order of $H$ is 72, and so $H$ has index 2 in $\Gamma(9)$. 
Let $A$ be the unique, nontrivial, abelian, normal subgroup of $H$. 
Then $A$ has order $9$. 
We claim that $A = T$. 
On the contrary, suppose $A \neq T$. 
Then $A\cap T$ is an abelian normal subgroup of $H$ properly contained in $A$. 
Therefore $A\cap T =\{x\}$, and so $A$ injects into the group $\Gamma(9)/T$ of order 16, 
which is not the case, since $A$ has order 9. 
Therefore $A = T$ as claimed. 

The group $G_1^\Omega$ acts transitively on $\Omega -\{C\}$ by Theorem 9.1 of \cite{W}. 
Hence, the induced action of $G_1^\Omega$ on $N-\{1\}$ by conjugation is transitive by Theorem 11.2 of \cite{W}. 
Therefore, the induced action of $H$ on $T-\{x\}$ by conjugation is transitive. 
Now the induced action of $T$ on $T-\{x\}$ by conjugation is trivial, since $T$ is abelian. 
The group $H\cap \Gamma_0(9)$ is a set of coset representatives for $T$ in $H$, 
and so the induced action of $H\cap \Gamma_0(9)$ on $T-\{x\}$ by conjugation is transitive. 
As $x^3$ acts trivially on $x+1$ by conjugation and $(ax)(x+1)(ax)^{-1} = x + a$, we conclude that
$$\{a \in F_9^*: ax^{3^c} \in H\ \hbox{for some}\ c = 0, 1\} = F_9^*.$$
The group $F_9^*$ is cyclic by Theorem 11 of \cite{L}.

Assume that $H$ contains $ax$ for some generator $a$ of $F_9^*$. 
Then $H\cap \Gamma_0(9)$ and $H/T$ are cyclic groups of order 8. 
Now assume that $H$ does not contain $ax$ for some generator $a$ of $F_9^*$. 
Suppose that $a$ is a generator of $F_9^*$. 
Then the generators of $F_9^*$ are $a, a^3, a^5, a^7$, and so $H$ contains $ax^3, a^3x^3, a^5x^3, a^7x^3$. As $(ax^3)^2 = a^4x$, $(ax^3)(a^3x^3)= a^2x$, $(ax^3)(a^5x^3) = x$ and $(ax^3)(a^7x^3) = a^6x$, 
we have that 
$$H\cap \Gamma_0(9) =\{x, ax^3, a^2x, a^3x^3, a^4x,a^5x^3, a^6x,a^7x^3\}.$$
Now $\langle ax^3\rangle = \{x, ax^3, a^4x, a^5x^3\}$ and $\langle a^2x\rangle = \{x, a^2x, a^4x, a^6x\}$ are  distinct cyclic groups of order 4. 
Therefore $H\cap\Gamma_0(9)$ is a quaternion group. 
Hence $H/T$ and $G_1^\Omega/N$ are quaternion groups. 
Thus, in either case $G_1^\Omega/N$ is not a dihedral group. 

Now, the group $G_1/(K\cap G_1)$ projects onto $G_1^\Omega/N$ and 
$G_1/(K\cap G_1)$ is either a cyclic group of order at most 6 or a dihedral group of order 
4, 6, 8 or 12.  
As $G_1^\Omega/N$ is not a dihedral group, $G_1^\Omega/N$ must be the image 
of a dihedral group of order 12. 
As 8 does not divide 12, we have a contradiction. 
Therefore, the case $m = 2$ and $p = 3$ does not occur. 
\end{proof}

\section{Examples} 

In this section, we consider examples of multiply transitive cusped hyperbolic 3-manifolds 
that satisfy Theorems 2, 3, and 4.  All the examples are orientable. 
Most of our examples are tessellated by isometric, ideal, regular polyhedra so that 
any face of a polyhedron is the face of exactly two polyhedra of the tessellation. 
Such a tessellation is called a {\it Platonic tessellation}. 

A {\it flag} of a Platonic tessellation is a triple $(P, F, E)$ consisting of a polyhedron $P$ of the tessellation, a face $F$ of $P$, and an edge $E$ of $F$.  We say that an isometry $\phi$ acts on a flag $(P, F, E)$ 
if $(\phi(P), \phi(F), \phi(E))$ is also a flag. 

\vspace{.15in}
\noindent{\bf Example 1.} According to Thurston \cite{T}, the complement in $S^3$ of the minimally twisted 5-link chain shown in Figure 1 is a 5-transitive hyperbolic 3-manifold $M$ of volume 10.149 \ldots . 
SnapPy \cite{C-W} shows that $M$ has a Platonic tessellation consisting of 10 ideal regular tetrahedra. 
The group $G$ of isometries of $M$ acts transitively on the set of flags of this tessellation 
so that the stabilizer of each flag has order 2. 
Hence $G$ has order 240.  The group $G$ is $(\integers/2\integers) \times S_5$. The generator of the $\integers/2\integers$ factor is orientation-preserving and acts trivially on the set $\Omega$ of cusps of $M$, and the $S_5$ factor acts 5-transitively on $\Omega$. 
The group of orientation-preserving isometries of $M$ acts only 3-transitively on $\Omega$.  

\begin{figure}[t] 
\centering
{\includegraphics[width=1.6in]{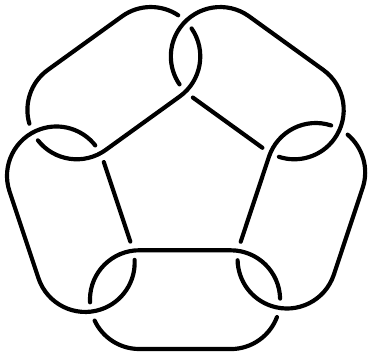}}
\caption{A minimally twisted 5-link chain}
\end{figure}

\vspace{.15in}
\noindent{\bf Example 2.} According to Vogeler \cite{V}, the complement in $S^3$ of the link shown in Figure 2 is a 4-transitive hyperbolic 3-manifold $M$.  Using SnapPy \cite{C-W}, we found that $M$ has volume 24.092 \ldots and the group $G$ of isometries of $M$ is $(\integers/2\integers) \times S_4$. 
The $\integers/2\integers$ factor of $G$ acts trivially on $\Omega$,  
and the $S_4$ factor of $G$ acts 4-transitively on $\Omega$. 

\begin{figure}[t] 
\centering
{\includegraphics[width=1.6in]{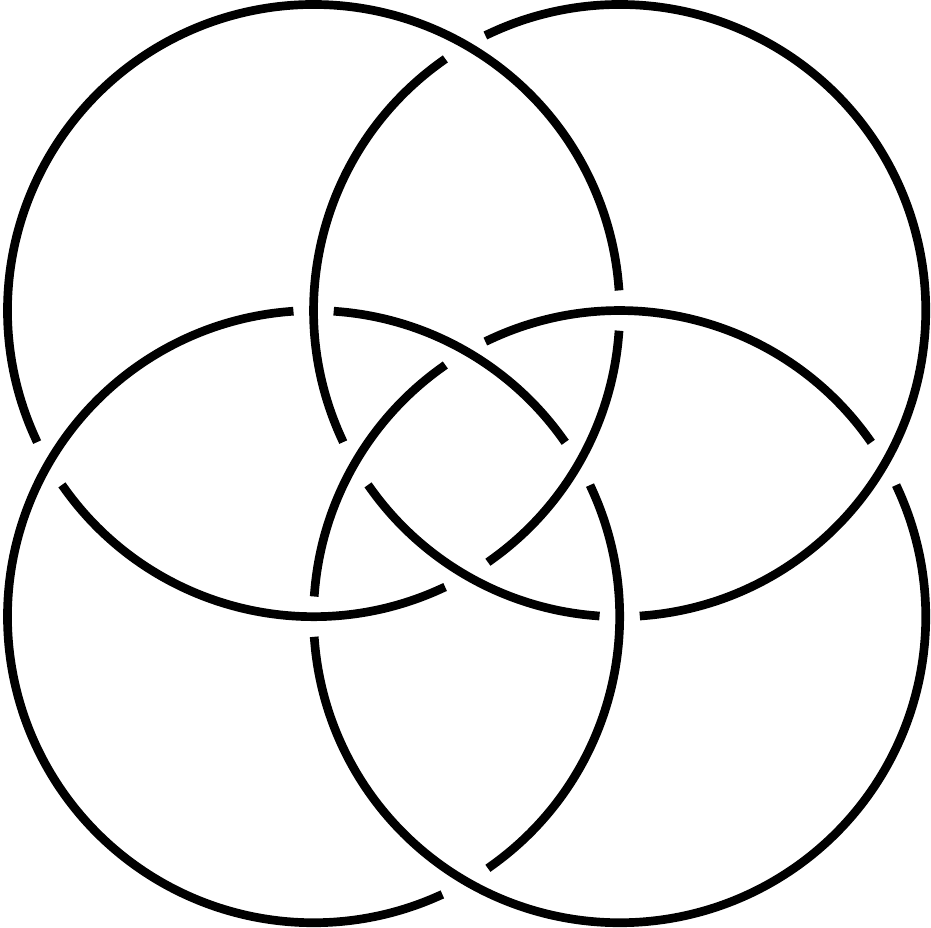}}
\caption{A link whose complement in $S^3$ is 4-transitive}
\end{figure}

\vspace{.15in}
\noindent{\bf Example 3.}
According to Thurston \cite{T79}, the complement in $S^3$ of the Borromean rings shown in Figure 3 is 
a hyperbolic 3-manifold $M$ of volume 7.327 \ldots that has a Platonic tessellation consisting of 2 ideal regular octahedra.  
The group $G$ of isometries of $M$ acts freely and transitively on the set of flags of this 
tessellation.  Hence $G$ has order 48.  

The group $G$ is a semi-direct product 
$(\integers/2\integers)^3 \rtimes S_3$ with $S_3$ acting 3-transitively 
on the standard basis of the vector space $(\integers/2\integers)^3$. 
The subgroup of $G$ that stabilizes both octahedra is 
the semi-direct product $(\integers/2\integers)^3 \rtimes C_3$ where $C_3$ is the cyclic subgroup of $S_3$ 
of order 3. 

If we consider the ideal regular octahedron $P$ in the conformal ball model $B^3$ of hyperbolic 3-space with ideal vertices $\pm e_1, \pm e_2, \pm e_3$, where $\{e_1, e_2, e_3\}$ is the standard basis of $\realnos^3$, then the action of the standard basis vectors of $(\integers/2\integers)^3$ on either octahedra of the tessellation of $M$ correspond 
to the action of the reflections in the coordinate hyperplanes of $B^3$ on $P$. 

The elementary group $(\integers/2\integers)^3$ acts trivially on the set $\Omega$ of cusps of $M$.  
The quotient group $G^\Omega = S_3$ acts 3-transitively on $\Omega$.

\begin{figure}[b] 
\centering
{\includegraphics[width=1.6in]{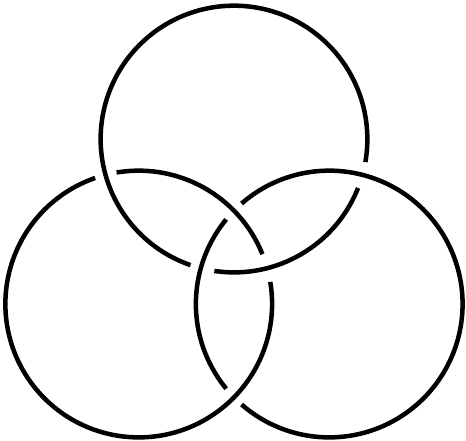}}
\caption{The Borromean rings}
\end{figure}

\vspace{.15in}
\noindent{\bf Example 4.} According to Goerner \cite{G}, the complement in $S^3$ of the link shown in Figure 4 is a hyperbolic 3-manifold $M$ that has a Platonic tessellation consisting of 5 ideal regular octahedra. 
The volume of $M$ is $18.319 \ldots .$ 
The group $G$ of isometries of $M$ acts freely and transitively on the set of flags of this tessellation. 
Hence $G$ has order 120. Every isometry of $M$ is orientation-preserving. 
The group $G$ acts effectively and sharply 3-transitively on the set of 6 cusps of $M$, 
and therefore $G = {\rm PGL}(2,5)$ by Theorem 2.6 of Chapter XI of \cite{H-B}. 
Although $G \cong S_5$, we write $G = {\rm PGL}(2,5)$ to indicate that $G$ is a degree 6 permutation group according to its natural representation as a permutation group on the set of the six 1-dimensional vector subspaces of the vector space $F_5\oplus F_5$ over the field $F_5$.

\begin{figure}[t] 
\centering
{\includegraphics[width=1.6in]{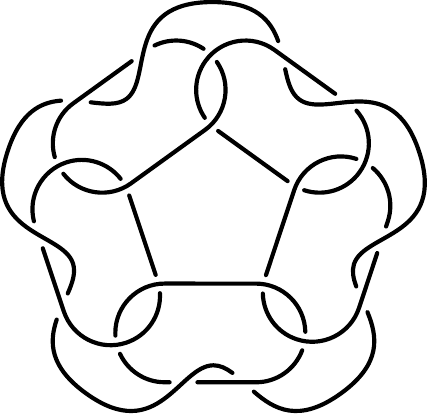}}
\caption{The congruence link $(1,\langle 2 + i\rangle)$}
\end{figure}

\vspace{.15in}
\noindent{\bf Example 5.} According to Thurston \cite{T}, the complement in $S^3$ of the link shown in Figure 5 is a hyperbolic 3-manifold $M$ with extraordinary 336-fold symmetry and volume $28.418 \ldots$ . 
According to Goener \cite{G}, the manifold $M$ has a Platonic tessellation consisting of 28 ideal regular tetrahedra, and the group $G$ of isometries of $M$ acts freely and transitively on the set of flags of this 
tessellation. Hence $G$ has order 336.  Every isometry of $M$ is orientation-preserving.  
The group $G$ acts effectively and sharply 3-transitively on the set of 8 cusps of $M$,  and so $G = {\rm PGL}(2,7)$ by Theorem 2.6 of Chapter XI of \cite{H-B}. 

\begin{figure}[b] 
\centering
{\includegraphics[width=1.6in]{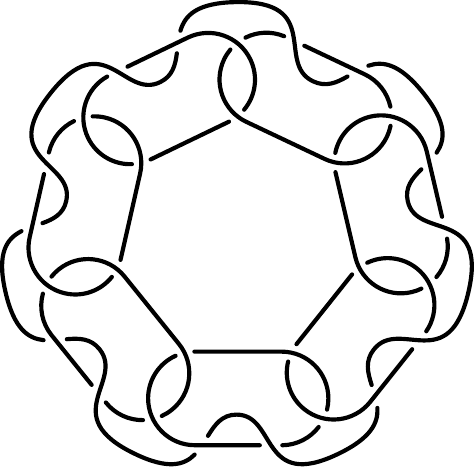}}
\caption{Thurston's congruence link}
\end{figure}

\vspace{.15in}
\noindent{\bf Example 6.} 
Our last example is an orientable hyperbolic 3-manifold $M$ that is obtained by gluing together two ideal regular dodecahedra along their faces.  Each face of the first dodecahedron is glued to the opposite face of the second dodecahedron with a twist of $\pi/5$ as in Figure 10.1.1 of \cite{R}.  
The manifold $M$ has volume $41.160 \ldots$ and exactly 5 cusps. 
The group $G$ of isometries of $M$ acts freely and transitively on the set of flags of this Platonic tessellation. Hence $G$ has order 120.  Every isometry of $M$ is orientation-preserving. 
This implies that $M$ is manifold {\bf odode02\_00912} in Goerner's census of double-dodecahedral cusped hyperbolic 3-manifolds (see Tables 4 and 10 of \cite{G2}). 
The manifold $M$ is a homology link, 
but $M$ is not the complement of a link in $S^3$ by the main theorem of Goener's paper \cite{G}.

The group $G$ is $(\integers/2\integers)\times A_5$.  
The subgroup of $G$ that stabilizes both dodecahedra is the $A_5$ factor of $G$. 
The generator of the $\integers/2\integers$ factor of $G$ acts trivially on the set $\Omega$ of cusps of $M$, and the $A_5$ factor of $G$ acts effectively and sharply 3-transitively on $\Omega$.

\section{Final Remarks}

We conclude with some natural problems that this paper suggests. 

\vspace{.15in}
\noindent{\bf Problem 1:} 
Classify all the hyperbolic 3-manifolds that satisfy Theorems 2, 3, or \nolinebreak 4.  

\vspace{.15in}
\noindent In this regard, we ask the following more specific questions:

\vspace{.15in}
\noindent{\bf Question 1:} Is the complement in $S^3$ of the minimally twisted 5-link chain the only hyperbolic 3-manifold that satisfies Theorem 2?

\vspace{.15in}
\noindent{\bf Question 2:} Are there any nonorientable hyperbolic 3-manifolds that satisfy 
Theorems 2, 3, or 4?

\vspace{.15in}
\noindent{\bf Problem 2:}
Determine the possible numbers of cusps of $2$-transitive cusped hyperbolic 3-manifolds of finite volume. 

\vspace{.15in}
\noindent According to Vogeler \cite{V},  there exists a $2$-transitive hyperbolic 3-manifold of finite volume with exactly $q$ cusps for each prime power $q$. 

\vspace{.15in}
\noindent{\bf Acknowledgments} 

\vspace{.15in}
\noindent We thank Matthias Goerner for providing us with graphic files for Figures 1, 4, and 5,  
that first appeared in \cite{B-G-R}, and for helpful communications concerning 
his hyperbolic 3-manifold {\bf odode02\_00912}.

\end{document}